\documentclass{amsart}
\usepackage{amsfonts}
\usepackage{amsmath}
\usepackage{amssymb}
\usepackage{graphicx}
\usepackage{version}
\usepackage{caption}
\newcommand{\R}{{\mathbb R}}
\theoremstyle{plain}
\newtheorem{theorem}{Theorem}[section]
\newtheorem{corollary}{Corollary}[section]

\newtheorem{proposition}{Proposition}[section]
\theoremstyle{definition}
\newtheorem{definition}{Definition}[section]
\newtheorem{example}{Example}[section]
\newtheorem{remark}{Remark}[section]
\numberwithin{equation}{section}
\excludeversion{comment1}
\begin{document}
\title{Approximation of rough functions}
\author[M. F. Barnsley]{M. F. Barnsley}
\address{Australian National University }
\author[B. Harding]{B. Harding}
\address{Australian National University }
\author[A. Vince]{A. Vince}
\address{University of Florida, Australian National University}
\author[P. Viswanathan]{P. Viswanathan}
\address{Australian National University }

\begin{abstract}
For given $p\in\lbrack1,\infty]$ and $g\in L^{p}\mathbb{(R)}$, we establish
the existence and uniqueness of solutions $f\in L^{p}(\mathbb{R)}$, to the
equation
\[
f(x)-af(bx)=g(x),
\]
where $a\in\mathbb{R}$, $b\in\mathbb{R} \setminus \{0\}$, and $\left\vert
a\right\vert \neq\left\vert b\right\vert ^{1/p}$. Solutions include well-known
nowhere differentiable functions such as those of Bolzano, Weierstrass, Hardy,
and many others. Connections and consequences in the theory of fractal
interpolation, approximation theory, and Fourier analysis are established.

\end{abstract}
\maketitle

\section{Introduction}

The subject of this paper, in broad terms, is fractal analysis. More
specifically, it concerns a constellation of ideas centered around the single
unifying functional equation~\eqref{eqref} below. In practice, the given
function $g(x)$ may be smooth and the solution $f(x)$ is often rough,
possessing fractal features. Classical notions from interpolation and
approximation theory are extrapolated, via this equation, to the fractal
realm, the basic goal being the utilization of fractal functions to analyze
real world rough data.

For given $p\in\lbrack1,\infty]$ and $g:\R\rightarrow \R$ with $g\in L^{p}\mathbb{(R)}$, we establish
the existence and uniqueness of solutions $f\in L^{p}(\mathbb{R)}$, to the
equation
\begin{equation}
f(x)-af(bx)=g(x),\label{eqref}
\end{equation}
where $a\in\mathbb{R}$,  $b\in\mathbb{R} \setminus \{0\}$, and $\left\vert
a\right\vert \neq\left\vert b\right\vert ^{1/p}$. By uniqueness we mean that
any solution is equal to $f$ almost everywhere in $\mathbb{R}$. When $a$, $b$
and $g$ are chosen appropriately, solutions include the classical nowhere
differentiable functions of Bolzano, Weierstrass, Hardy, Takagi, and others;
see the reviews \cite{baranski2, hunt}. For example, the continuous, nowhere
differentiable function presented by Weierstrass in 1872 to the Berlin
Academy, defined by
\begin{equation}
f(x)=\sum_{k=0}^{\infty}a^{k}\cos\,(\pi b^{k}x),\label{eq:W}
\end{equation}
where $0<a<1$, $b$ is an integer, and $ab\geq1+\frac{3}{2}\pi$ (see
\cite{Hardy}), is a solution to the functional equation~\eqref{eqref} when
$g(x)=\cos(\pi x)$. The graph of $f$ was studied as a fractal curve in the
plane by Besicovitch and Ursell \cite{BU}. An elementary and readable account
of the history of nowhere differentiable functions is \cite{thim}; it includes
the construction by Bolzano (1830) of one of the earliest examples of such a function.
Analytic solutions to the functional equation~\eqref{eqref} for various values of $a$ and $b$,
when $g$ is analytic, have been studied by Fatou in connection with Julia sets \cite{fatou,Read}.

If $\left\vert ab\right\vert >1$, $b>1$ is an integer, and $g$ has certain
properties, see \cite{baranski2, hunt}, then the graph of $f$, restricted to
$[0,1]$, has box-counting (Minkowski) dimension
\[
D=2+\frac{\ln\left\vert a\right\vert }{\ln b}\text{.}
\]
In particular, if $g(x)=\cos(\pi x)$, then by a recent result of
B\'{a}r\'{a}ny, Romanowska, and Bara\'{n}ski \cite{baranski} the Hausdorff
dimension of the graph of $f$ is $D$, for a large set of values of $\left\vert
a\right\vert <1$.

Notation that is used in this paper is set in Section~\ref{sec:notation}. In
Section~\ref{sec:key} we establish existence and uniqueness of solutions to
Equation~\eqref{eqref} in various function spaces (see Theorem~\ref{thm:key}, Corollaries~\ref{keycor} and \ref{cor:key},
Proposition~\ref{prop:spaces}).
Although the emphasis has been on the pathology of the solution to the functional
equation~\eqref{eqref}, it is shown that, if $g$ is continuous, then the solution $f$ is continuous
(see Corollaries~\ref{ctycor} and \ref{cor:cont}).

A widely used method for constructing fractal sets, in say ${\mathbb{R}}^{2}$, is as
the attractor of an iterated function system (IFS). Indeed, starting in the
mid 1980's, IFS fractal attractors $A$ were systematically constructed so that
$A$ is the graph of a function $f\,:J\rightarrow{\mathbb{R}}$, where $J$ is a
closed bounded interval on the real line \cite{Barnsley1}.  Moreover $f$ can be
made to interpolate a data set of the form $(x_{0},y_{0}),(x_{1},y_{1}),\dots,(x_{N},y_{N})$,
where $x_{0}<x_{1}<\cdots<x_{N}$ and $J=[x_{0},x_{N}]$.
The book \cite{Mas} is a reference on such fractal interpolation functions
constructed via an IFS.  One of the appeals of the theory is that it is
possible to control the box-counting dimension and smoothness of the graphs of
such functions.  The solutions to the functional equation~\eqref{eqref} include,
not only the classical nowhere differentiable functions, but also  fractal interpolation
functions. This is the subject of  Section~\ref{sec:interpolation}, in particular Theorems~\ref{thm:interp1} and \ref{thm:interp2}.
One impetus for the research reported here is the work on
fractal interpolation by Massopust \cite{Mas}, Navascu\`{e}s \cite{Nav1, Nav2,
manuvascues1}, and Chand and his students \cite{CK}.

In Section~\ref{sec:ON}, Equation~\eqref{eqref},
and the theory surrounding it are leveraged to obtain orthogonal expansions
--that we call Weierstrass Fourier series-- and corresponding approximants,
for various functions, both smooth and rough, using approximants with
specified Minkowski and even Hausdorff dimension.

Some ideas in the present work are anticipated, at least in flavor, in Deliu
and Wingren \cite{deliu1} and Kigami and his collaborators
\cite{kigami2,yamaguti}. But, as far as we know, our main observations,
 namely Theorem~\ref{thm:key} and its corollaries, Theorems~\ref{thm:interp1} and \ref{thm:interp2}, and Theorem~\ref{thm:ONbasis}, are new.

\section{Notation}

\label{sec:notation}

For $p\in\lbrack1,\infty),$ $L^{p}(X)$ denotes the Banach space of functions
$f:X\rightarrow\mathbb{R}$ such that
\[
\int_{X}|f(x)|^{p}dx<\infty\text{,}
\]
where the integration is with respect to Lebesgue measure on $X$. In this
paper, $X$ will be ${\mathbb{R}}$, a closed interval of ${\mathbb{R}}$, or an
interval of the form $[c,\infty),\,(-\infty,c]$ or $(-\infty,c]\cup\lbrack
c^{\prime},\infty$). The space $L^{\infty}(X)$ denotes the Banach space of
functions $f:X\rightarrow\mathbb{R}$ such that the essential supremum of
$\left\vert f\right\vert $ is bounded. For all $p\in\lbrack0,\infty]$, the
norm of $f\in$ $L^{p}(X)$ is denoted $\left\Vert f\right\Vert _{p}$, where
\begin{align*}
\left\Vert f\right\Vert _{p}  &  =\left\{  \int_{X}|f(x)|^{p}dx\right\}
^{1/p}\text{ when }p\in\lbrack1,\infty),\\
\left\Vert f\right\Vert _{\infty}  &  =\inf\{M\in\lbrack0,\infty):\left\vert
f(x)\right\vert \leq M\text{ for almost all }x\in X\}.
\end{align*}

The norm of a bounded linear operator $H:L^{p}(X)\rightarrow L^{p}(X)$ is
defined by
\[
\left\Vert H\right\Vert _{p}=\max\{\left\Vert Hf\right\Vert _{p}:\left\Vert
f\right\Vert _{p}=1\}.
\]
 The space of bounded uniformly continuous
real valued functions with the supremum norm is denoted $C(X)$.
Further let
\[C^{k}(X) = \{ f \, : \, f^{(j)} \in C(X), \, j=0,1,\dots, k\}.\]
For a bounded
continuous function $f$ and $\alpha\in(0,1]$, let
\[
\lbrack f]_{\alpha}=\sup_{x,y\in X,~x\neq y}\frac{|f(x)-f(y)|}{|x-y|^{\alpha}
}.
\]
For $k\in\mathbb{N}\cup\{0\}$, the H\"older space
\[
C^{k,\alpha}({\mathbb{R}}):=\{ \, f\in C(\R) \, : \, f^{(j)}\in C({\mathbb{R)}}, \,
j=0,1,..,k, \, \Vert f\Vert_{C^{k,\alpha}}<\infty\}.
\]
where
\[
\Vert f\Vert_{C^{k,\alpha}}:=\sum_{j=0}^{k}\Vert f^{(j)}\Vert_{\infty
}+[f]_{\alpha}
\]
is a Banach space.

Let $k \ge1$ be an integer and $f \in L_{loc}^{1}({\mathbb{R}})$, the space of
all locally integrable functions. A function $g\in L_{loc}^{1}({\mathbb{R}})$
is a weak-derivative of $f$ of order $k$ if
\[
\int_{{\mathbb{R}}} g(x) \phi(x) dx = (-1)^{k} \int_{{\mathbb{R}}} f(x)
\phi^{k} (x) dx
\]
for all $\phi\in\mathcal{C}_{c}^{\infty}({\mathbb{R}})$, where $\mathcal{C}
_{c}^{\infty}({\mathbb{R}})$ is the space of continuous functions with compact
support, having continuous derivatives of every order.

For $1\leq p\leq\infty$ and $k\in\mathbb{N}\cup\{0\}$, let $W^{k,p}
(\mathbb{R)}$ denote the usual Sobolev space. That is,
\[
f\in W^{k,p}(\mathbb{R)}\Longleftrightarrow f^{(j)}\in L^{p}(\mathbb{R}
),~j=0,1,\dots,k,
\]
where $f^{(j)}$ denotes the $j$-th weak or distributional derivative of $f$.
The space $W^{k,p}({\mathbb{R}})$ endowed with the norm
\[
\Vert f\Vert_{W^{k,p}}:=\left\Vert f^{(k)}\right\Vert _{p}+\left\Vert
f\right\Vert _{p}
\]
is a Banach space.

Consider the difference operator
\[
\Delta_{h} f(x) = f(x-h)-f(x)
\]
and define the modulus of continuity by
\[
\omega_{p}^{2} (f,t) = \sup_{|h| \le t} \Vert\Delta_{h}^{2} f \Vert_{p}.
\]
For $n \in\mathbb{N}\cup\{0\}$, $s=n+ \alpha$, $0 < \alpha\le1$ and $1\le
p,q\le\infty$, the Besov space $B_{p,q}^{s} ({\mathbb{R}})$ consists of all
functions $f$ such that
\[
f \in W^{n,p} ({\mathbb{R}}), ~ \int_{0}^{\infty}\Big \vert \frac{w_{p}
^{2}(f^{(n)},t)}{t^{\alpha}}\Big \vert^{q} \frac{dt}{t} < \infty.
\]
The functional
\[
\Vert f\Vert_{B_{p,q}^{s}} := \Big[\Vert f\Vert_{W^{n,p}}^{q} + \int
_{0}^{\infty}\Big \vert \frac{w_{p}^{2}(f^{(n)},t)}{t^{\alpha}}\Big \vert^{q}
\frac{dt}{t}\Big]^{\frac{1}{q}}
\]
is a norm which turns $B_{p,q}^{s} ({\mathbb{R}})$ into a Banach space.

\section{Key Theorem}

\label{sec:key}

The functional equation~(\ref{eqref}) can be expressed as
\[ M_{a,b} \, f = g,\] where the linear operator $M_{a,b}$ is defined as follows.

\begin{definition}
\label{def:operators} For all $p\in\lbrack1,\infty],\,a,b\in{\mathbb{R}
},\,b\neq0$, the linear operators $T_{b}:L^{p}(\mathbb{R)\rightarrow}
L^{p}(\mathbb{R)}$ and $M_{a,b}:L^{p}(\mathbb{R)\rightarrow}L^{p}(\mathbb{R)}$
are given by
\[
\begin{aligned}
(T_{b}f)(x)&=f(bx) \\M_{a,b}f &=(I-aT_{b})f
\end{aligned}
\]
for all $x\in\mathbb{R}$ and all $f\in L^{p}({\mathbb{R}})$. By convention, if
$p=\infty$ and $b\neq0$, then $|b|^{\frac{1}{p}}=1$.
\end{definition}

\begin{theorem}
\label{thm:key} For all $p\in\lbrack1,\infty], \, a, b \in{\mathbb{R}}, \,
b\neq0$, and $\left\vert a\right\vert \neq\left\vert b\right\vert ^{\frac
{1}{p}}$, the linear operators $T=T_{b}$ and $M = M_{a,b}$ are homeomorphisms
from $L^{p}(\mathbb{R)}$ to itself. In particular,

\begin{enumerate}
\item
\[
T_{b}^{-1}=T_{1/b}
\]

\item
\[
\left\Vert T_{b}\right\Vert _{p}={\left\vert b\right\vert ^{-\frac{1}{p}} }
\]

\item
\[
\left|  1-\frac{\left\vert a\right\vert }{\left\vert b\right\vert ^{\frac
{1}{p}} } \right|  \left\Vert f \right\Vert _{p} \, \leq\, \left\Vert M_{a,b}f
\right\Vert _{p} \, \leq\, \left(  1+\frac{\left\vert a\right\vert
}{\left\vert b\right\vert ^{\frac{1}{p}}} \right)  \left\Vert f \right\Vert
_{p}
\]

\item
\[
{M_{a,b}}^{-1} =
\begin{cases}
\quad\sum_{n=0}^{\infty}\, a^{n}\, T_{b}^{n} & \text{ if } \; |a|<\left\vert
b\right\vert ^{\frac{1}{p}},\\
& \\
-\sum_{n=1}^{\infty} \, \left(  \frac{1}{a}\right)  ^{n}\,T_{1/b}^{n} & \text{
if } \; |a|>\left\vert b \right\vert ^{\frac{1}{p}}.
\end{cases}
\]

\end{enumerate}
\end{theorem}

\begin{proof}
It is readily verified that $T$ is invertible with inverse $T_{b}^{-1}
=T_{1/b}$ and that the formula (2) for the $p$-norm of $T_{b}$ holds.
Consequently
\[
\left\Vert T_{b}^{-1}\right\Vert _{p}=\left\vert b\right\vert ^{\frac{1}{p}
}\text{.}
\]
To show that $M$ is injective, assume that $Mf=0,$ i.e., $\left\Vert
f-aT_{b}f\right\Vert _{p}=0$. If $\left\vert a\right\vert <\left\vert
b\right\vert ^{\frac{1}{p}}$, then
\begin{align*}
0  &  =\left\Vert f-aT_{b}f\right\Vert _{p}\geq\left\Vert f\right\Vert
_{p}-\left\vert a\right\vert \left\Vert T_{b}\right\Vert _{p}\left\Vert
f\right\Vert _{p}=(1-\left\vert a\right\vert \left\Vert T_{b}\right\Vert
_{p})\left\Vert f\right\Vert _{p}\\
&  =\left(  1-\left\vert a\right\vert \left\vert b\right\vert ^{-\frac{1}{p}
}\right)  \left\Vert f\right\Vert _{p}\geq0\text{,}
\end{align*}
which implies that $\left\Vert f\right\Vert _{p}=0$. Similarly, if $\left\vert
a\right\vert >\left\vert b\right\vert ^{\frac{1}{p}}$, then $a\neq0$, and
using the invertibility of $T_{b}$, we have that $Mf=0$ if and only if
$a^{-1}\,(T_{b})^{-1}f=f$, which implies
\begin{align*}
0  &  =\left\Vert a^{-1}T_{b}^{-1}f-f\right\Vert _{p}\geq(1-\left\vert
a\right\vert ^{-1}||T_{b}^{-1}||)\Vert f\Vert_{p}\\
&  =(1-\left\vert a\right\vert ^{-1}\left\vert b\right\vert ^{\frac{1}{p}
})\left\Vert f\right\Vert _{p}\geq0\text{,}
\end{align*}
which in turn implies that $\left\Vert f\right\Vert _{p}=0$. Inequality (3)
follows from (2) and the triangle inequality.

To show that $M$ is surjective and that a solution to $Mf=g$ in $L^{p}
(\mathbb{R})$ is
\[
f=
\begin{cases}
\quad\sum_{n=0}^{\infty}\,a^{n}T_{b}^{n}\,g\quad\qquad\text{if}\quad
|a|<\left\vert b\right\vert ^{\frac{1}{p}},\\
\\
-\sum_{n=1}^{\infty}\, \left(  \frac{1}{a}\right)  ^{n}\,T_{1/b}^{n}
\,g\quad\text{if}\quad|a|>\left\vert b\right\vert ^{\frac{1}{p}},
\end{cases}
\]
first note that both series are absolutely and uniformly convergent in
$L^{p}(\mathbb{R)}$. This is because their partial sums are Cauchy sequences.
Now, using the continuity of $M:L^{p}(\mathbb{R)\rightarrow}L^{p}(\mathbb{R)}$
and equality (2) in the statement of the theorem, we have, in the case
$|a|<\left\vert b\right\vert ^{\frac{1}{p}}$,
\[
M\left(  \sum_{n=0}^{\infty}\,a^{n}T_{b}^{n}\,g\right)  =\lim_{k\rightarrow
\infty}\left(  \sum_{n=0}^{k}\,a^{n}MT_{b}^{n}\,g\right)  =\lim_{k\rightarrow
\infty}(I-a^{k+1}T_{b}^{k+1})g=g.
\]
A similar derivation holds in the case $|a|>\left\vert b\right\vert ^{\frac
{1}{p}}$.
\end{proof}

The next corollary on existence and uniqueness of solutions to
Equation~\eqref{eqref} follows at once from Theorem~\ref{thm:key}.

\begin{corollary}
\label{keycor} Assume that $a, b\in\mathbb{R}, \, b\neq0$, and $\left\vert
a\right\vert \neq\left\vert b\right\vert ^{\frac{1}{p}}$. For any $g\in
L^{p}(\mathbb{R)}, \, p\in\lbrack1,\infty]$, there is a unique solution $f\in
L^{p}(\mathbb{R)}$ to the equation
\[
f(x)-af(bx)=g(x),
\]
and the solution is given by the following series that are absolutely and
uniformly convergent in $L^{p}({\mathbb{R}})$:
\begin{equation}
\label{eq:sum}f(x)=
\begin{cases}
\quad\sum_{n=0}^{\infty}\,a^{n}\,g(b^{n}x)\qquad\text{if}\quad|a|<\left\vert
b\right\vert ^{\frac{1}{p}}\\
\\
-\sum_{n=1}^{\infty}\,\left(  \frac{1}{a})^{n}\,g(\frac{x}{b^{n}} \right)
\quad\text{if}\quad|a|>\left\vert b\right\vert ^{\frac{1}{p}}.
\end{cases}
\end{equation}

\end{corollary}

\begin{remark}
Recall that the adjoint of a bounded linear operator $A: X \to Y$ is the
operator $A^{\ast}: Y^{\ast} \to X^{\ast}$ defined by $(A^{\ast}\mu) (x)=
\mu(Ax)$ for all $\mu\in Y^{\ast}$ and $x \in X$, where $X^{\ast}$ denotes the
dual space of $X$. For $1 \le p < \infty$ there is a canonical isomorphism
between $L^{p}({\mathbb{R}})^{*}$ and $L^{q} ({\mathbb{R}})$, where $\frac
{1}{p}+\frac{1}{q}=1$. For each linear functional $\mu\in L^{p} ({\mathbb{R}
})^{\ast}$, this isomorphism associates a unique representative $g \in L^{q}
({\mathbb{R}})$ such that $\mu(f) = \int_{{\mathbb{R}}} f(x) g(x)~dx$ for all
$f\in L^{p}({\mathbb{R}})$. It is routine to show that
\[
T_{b}^{\ast}=\frac1b \, T_{\frac1b}
\]
in the sense that, if the representative of $\mu\in L^{p} ({\mathbb{R}}
)^{\ast}$ in the space $L^{q}({\mathbb{R}})$ is $g$, then the representative
of $T_{b}^{\ast} \mu$ is $b^{-1}T_{b^{-1}} g$. Similarly
\[
M_{a,b}^{\ast}=M_{\frac{a}{b}, \frac{1}{b}}.
\]

\end{remark}

\begin{remark}
\label{zerorem} If $b=0,$ then Equation~\eqref{eqref} has solution
\[
f(x)=g(x)+\frac{a}{1-a}\,g(0).
\]
So, for all $a\neq1$, there is a well-defined solution $f(x)$ for all
$x\in\mathbb{R}$, for each specified value of $g(0)$.  Since, as a element of
$L^{p}(\mathbb{R)}$, the function $g$ is defined only up to a set of measure $0$, the value $g(0)$ has little meaning.
Thus it does not make sense to consider Equation~\eqref{eqref} in $L^{p}
(\mathbb{R)}$ when $b=0$. However, the problem of finding $f$ for a given $g$ is
well-posed in spaces such as $C(\mathbb{R)}$, even when $b=0$.
\end{remark}

In view of Remark \ref{zerorem},\textit{ }except where otherwise stated, it is
assumed throughout this paper that $b\neq0$
and $\left\vert a\right\vert\neq\left\vert b\right\vert ^{\frac{1}{p}}$.
The results in Theorem~\ref{thm:key} and its Corollary~\ref{keycor} hold for
various spaces related to the $L^{p}$-spaces. Corollaries~\ref{ctycor},
\ref{cor:cont}, and \ref{cor:key} below concern these related spaces.

\begin{corollary}
\label{ctycor} Assume that $a, b\in\mathbb{R}, \, b\neq0$, and $\left\vert
a\right\vert \neq1$. For any $g\in C({\mathbb{R}})$, there is a unique
solution $f\in C({\mathbb{R}})$ to the equation
\[
f(x)-af(bx)=g(x),
\]
and the solution is given by the following series that are absolutely and
uniformly convergent:
\[
f(x)=
\begin{cases}
\quad\sum_{n=0}^{\infty}\,a^{n}\,g(b^{n}x) & \text{if}\quad|a|<1\\
& \\
-\sum_{n=1}^{\infty}\,\left(  \frac{1}{a})^{n}\,g(\frac{x}{b^{n}} \right)  &
\text{if}\quad|a|>1.
\end{cases}
\]

\end{corollary}

\begin{proof}
If $\left\vert a\right\vert <1$, then $\left\vert \sum_{n=M}^{\infty}
\,a^{n}\,g(b^{n}x)\right\vert <\frac{\left\vert a\right\vert ^{M}
}{1-\left\vert a\right\vert }\left\Vert g\right\Vert _{\infty}$. Therefore $$\lim_{M \to \infty} \sup_{x \in \mathbb{R}} \left\vert \sum_{n=M}^{\infty}
\,a^{n}\,g(b^{n}x)\right\vert=0,$$ which implies $\sum_{n=0}^{\infty}\,a^{n}\,g(b^{n}x)$ is absolutely and uniformly convergent. Since $g \in C(\mathbb{R})$, it follows that the infinite sum is a
continuous function. That the series is a solution of the functional equation can be verified at once by substitution, see also Corollary \ref{keycor}. A similar argument applies in the case $\left\vert
a\right\vert >1$.
\end{proof}

The following relationships between the continuity of $f$ and the continuity
of $g$ follow as in Corollary~\ref{ctycor}.

\begin{corollary} \label{cor:cont}
For the equation $M_{a,b}\,f=g$ in $L^{\infty}({\mathbb{R}})$, if $|a|<1$,
then the following hold.

\begin{enumerate}
\item If $b >0$, then $f \in C([0, \infty))$ if and only if $g \in C([0,
\infty))$.

\item If $b \ge1$, then $f \in C([1, \infty))$ if and only if $g \in C([1,
\infty))$.

\item If $0<b\leq1$, then $f\in C([0,1])$ if and only if $g\in C([0,1])$.
\end{enumerate}
\end{corollary}

A similar set of statements hold when $C(X)$ is replaced by $C^{\prime}(X)$,
the set of functions in $C(X)$ with countably many discontinuities.

\begin{remark}
If, in Corollary \ref{ctycor}, $g\in L^{\infty}(\mathbb{R)}$ is assumed
piecewise continuous with countably many points of discontinuity, rather than
continuous, then it follows by a similar argument that the solution $f\in
L^{\infty}(\mathbb{R)}$ to $Mf=g$ is piecewise continuous with at most
countably many points of discontinuity.
\end{remark}

\begin{remark}
Examples related to fractal interpolation (see Example \ref{anexample}) show
that $f=M^{-1}g$  may be continuous on $[0,1]$ even if $g$  possess discontinuities.
\end{remark}

Unlike continuity, it is well-known from basic real
analysis that $f  = M_{a,b}^{-1}g$ may fail to be differentiable even if $g$ is
differentiable.  Vice versa, when $|a| >1$ and $g$ is continuous, $f$  may be more
differentiable than $g$.  Thus, in a general sense, for $|a|<1$, the mapping $M_{a,b}^{-1}$ is a ``roughing''
operation, and for $|a| >1$, it is a ``smoothing" operation.

The following estimate is worth mentioning.

\begin{proposition} \label{prop:spaces}
Consider the equation $M_{a,b}\,f=g$ for $f\in C([0,\infty))$, $|a|<1$ and
$b>0$. Then the uniform distance between $f$ and $g$ satisfies
\[
\Vert g-f\Vert_{\infty} =
\Vert g-M_{a,b}^{-1}g\Vert_{\infty}\leq\frac{|a|}{1-|a|}\Vert g\Vert_{\infty
},
\]
Consequently
\[
\Vert I-M_{a,b}^{-1}\Vert_{\infty}\leq\frac{|a|}{1-|a|}.
\]
\end{proposition}

\begin{proof}
Note that
\[
\begin{split}
\big|g(x)-M_{a,b}^{-1} g(x) \big| =  &  ~ \big |g(x)- \sum_{n=0}^{\infty}a^{n}
g(b^{n} x)\big|\\
\le &  ~ \sum_{n=1}^{\infty}|a|^{n} \|g\|_{\infty}\\
=  &  ~ \frac{|a|}{1-|a|} \|g\|_{\infty}.
\end{split}
\]
Therefore $\|g- f\|_{\infty}\le\frac{|a|}{1-|a|} \|g\|_{\infty}$,
proving the assertion.
\end{proof}

The term automorphism in the next corollary refers to a linear map that is a homeomorphism of a space to itself.  In pariticular, statement (6)
in the corollary is used in Section~\ref{sec:ON}.

\begin{corollary}
\label{cor:key} If $M=M_{a,b}$ is the operator of
Definition~\ref{def:operators}, with $|a|\neq|b|^{\frac{1}{p}}$, then $M$ is a
automorphism when considered as a mapping on

\begin{enumerate}
\item $L^{p}([0,\infty))$ or $L^{p}((-\infty,0])$ if $b>0$;

\item $L^{p}([1,\infty))$ or $L^{p}((-\infty,1])$ if $b>1$;

\item $L^{p}((-\infty,-1]\cup\lbrack1,\infty))$ if $|b|>1$;

\item $L^{p}([0,1])$ or $L^{p}([-1,0])$ if $0<b<1$;

\item $L^{p}([-1,1])$ if $0 < |b|\leq1$;

\item $L^{\infty}({\mathbb{R}})\,\cap\,L^{2}([0,1])\,\cap\,\mathcal{P}$ if
$b\in\mathbb{N},|a|<1$, where $\mathcal{P}$ is the set of functions on
${\mathbb{R}}$ of period $1$. This is the set of essentially bounded periodic
functions on ${\mathbb{R}}$ that are square integrable when restricted to
$[0,1]$. This is with the understanding that, for $g\in L^{\infty}
({\mathbb{R}})\,\cap\,L^{2}([0,1])\,\cap\,\mathcal{P}$, a representative of
$M^{-1}g$ can be chosen to be periodic of period $1$.
\end{enumerate}
\end{corollary}

\begin{proof}
(1) The space $L^{p}({\mathbb{R}})$ is the direct sum of two subspaces $L_{+}$
and $L_{-}$, the first consisting of functions which vanish over the negative
reals and the second consisting of functions which vanish over the positive
reals. Since each of these two subspaces is mapped into itself by $M$ and
since $M$ is bijective on $L^{p}(\mathbb{R)}$, it follows that $M$ restricted
to $L_{+}$ and $M$ restricted to $L_{-}$ are both bijective. The proofs of
(2)-(6) are similar, some using Corollary~\ref{keycor}.
\end{proof}

For appropriate values of $a$ and $b$, the operator $M_{a,b}$ also defines an
automorphism in some standard spaces of smooth functions that occur frequently
in various fields of analysis such as approximation theory, numerical
analysis, functional analysis, harmonic analysis, and in particular in
connection with PDEs. The proof is similar to that of Theorem \ref{thm:key},
and hence is omitted.

\begin{proposition}
For the operator $M_{a,b}$ specified in Definition~\ref{def:operators} the
following properties hold.

\begin{enumerate}
\item If $|a| < \min\big\{|b|^{\frac{1}{p}}, |b|^{\frac{1}{p}-k}\big\}$ or
$|a| > \max\big\{|b|^{\frac{1}{p}}, |b|^{\frac{1}{p}-k}\big\}$, then $M_{a,b}$
is an automorphism on Sobolev space $W^{k,p}({\mathbb{R}})$.

\item If $|a| < \min\big\{1, |b|^{-1}, |b|^{-2}, \dots, |b|^{-k},
|b|^{-\alpha} \big\} $ or $|a| > \max\big\{1, |b|^{-1}, |b|^{-2}, \dots,
|b|^{-k}, |b|^{-\alpha} \big\} $, then $M_{a,b}$ is an automorphism on
H\"{o}lder space $C^{k,\alpha}({\mathbb{R}})$.

\item If $|a| < \min\big\{|b|^{\frac{1}{p}}, |b|^{\frac{1}{p}-n},
|b|^{\frac{1}{p}-n-\alpha} \big\}$ or $|a| > \max\big\{|b|^{\frac{1}{p}},
|b|^{\frac{1}{p}-n}, |b|^{\frac{1}{p}-n-\alpha} \big\}$, then $M_{a,b}$ is an
automorphism on Besov space $B_{p,q}^{s} ({\mathbb{R}})$.
\end{enumerate}

In all the above cases ${M_{a,b}}^{-1} = \sum_{n=0}^{\infty}\, a^{n}\,
T_{b}^{n}$ for the first set of admissible values of parameters $a,b$ and
${M_{a,b}}^{-1} =-\sum_{n=1}^{\infty} \, \left(  \frac{1}{a}\right)
^{n}\,T_{1/b}^{n}$ for the second set of admissible values of parameters $a,b$.
\end{proposition}

\begin{remark}
A straightforward but useful consequence of the fact that $M_{a,b}^{-1}$ is
an automorphism on various spaces is the following. It is well known that
Schauder bases are preserved under an isomorphism. Consequently, if
$\{f_{n}\}_{n=1}^{\infty}$ is a Schauder basis for $X$, where $X$ is one of
the spaces $L^{p}({\mathbb{R}})$, $W^{k,p}({\mathbb{R}})$, $C^{k,\alpha
}({\mathbb{R}})$ or $B_{p,q}^{s}({\mathbb{R}})$, then $\{M_{a,b}^{-1}
f_{n}\}_{n=1}^{\infty}$ is a Schauder basis consisting of rough analogues of
the functions $\{f_{n}\}_{n=1}^{\infty}$. In particular, if $\{f_{n}
\}_{n=1}^{\infty}$ is an orthonormal basis for the Hilbert space
$L^{2}({\mathbb{R}})$ or $W^{k,2}({\mathbb{R}})$, then $\{M_{a,b}^{-1}
f_{n}\}_{n=1}^{\infty}$ is a Riesz basis for $L^{2}({\mathbb{R}})$ or
$W^{k,2}({\mathbb{R}})$. Some orthonormal bases consisting of rough functions
obtained via $M_{a,b}^{-1}$ are discussed in detail in Section \ref{sec:ON}.
\end{remark}

\section{Fractal Interpolation}

\label{sec:interpolation}

To illustrate how standard fractal interpolation theory fits into the functional equation
framework, consider a given set of data points $\left\{  (x_{n},y_{n}
)\right\}  _{n=0}^{N}\subset\mathbb{R}^{2},\,N>1,$  with $0 = x_0 <x_1 <x_2 \cdots < x_N =1$.
 At minimum what one
seeks is a function $f:[0,1]\rightarrow\mathbb{R},$ such that

\begin{enumerate}
\item $f$ interpolates the data, i.e., $f(x_{n}) = y_{n}, \, n= 0,1,\dots, N$;

\item there is an IFS $F = ({\mathbb{R}}^{2}; w_{1}, w_{2}, \dots, w_{N})$
whose attractor is the graph of the function $f$ on the interval $[0,1]$;

\item parameters can be varied to control continuity and differentiability of
$f$ and the Minkowski dimension of the graph of $f$.
\end{enumerate}

The IFS maps $w_n, \, n=1,2,\dots, N,$  that are studied extensively in fractal interpolation theory \cite{Barnsley1} are of the
form
\begin{equation} \label{eq:I1} w_n(x,y) = \big( L_n(x), F_n(x,y) \big),\end{equation}
where
\begin{equation} \label{eq:I2}  L_n(x) = a_nx+b_n, \qquad F_n(x,y) = \alpha_ny + g_n(x),\end{equation}
$|\alpha_n| <1$; $g_n \, : \, [0,1]\rightarrow \R$ is continuous; and
 \begin{equation} \label{eq:I3} \begin{aligned}
L_n(x_0) &=x_{n-1}, \\ L_n(x_N)&=x_n, \end{aligned} \qquad \qquad \begin{aligned} F_n(x_0,y_0)&=y_{n-1}, \\ F_n(x_N,y_N)&=y_n,
\end{aligned} \end{equation}
for all $n=1,2,\dots, N$.  In this case  there is a unique attractor of $F$, and it is the graph of a continuous  function $f$ that
interpolates the data \cite{Barnsley1}. The parameters $\alpha_n$ and $g_n$ can be varied to control continuity and differentiability of
$f$ and the Minkowski dimension of the graph of $f$.

We  specialize to
the uniform partition of $[0,1]$ and a constant scaling factor, i.e.,
\begin{equation} \label{eq:I4} L_n(x) = \frac{x+n-1}{N}, \qquad \qquad \alpha_n = a, \, |a| < 1, \end{equation}
 for all $n = 1,2,\dots, N$.
\vskip 2mm

The next two theorems make precise the close relationship between fractal interpolation functions and
solutions to the \textquotedblleft Weierstrass-type" functional equation.

\begin{theorem}  \label{thm:interp1}
  Given data points $\left\{  (x_{n},y_{n})\right\}  _{n=0}^{N}
\subset\mathbb{R}^{2},\,N>1,$ let $F$ be the IFS defined by the Equations~(\ref{eq:I1})-(\ref{eq:I4}),
and let $f^*$ be the function on $[0,1]$ whose graph is the attractor of $F$ and that interpolates the data.

Then $f^*$ is the unique solution to the functional equation $f(x)-af(Nx)=g(x)$ considered in the space $L^{\infty}(\R) \cap L^2([0,1]) \cap \mathcal P$ of
Corollary~\ref{cor:key}, where
\[ g(x) = \begin{cases}  g_n(L_n^{-1}(x)) & \text{if} \quad x \in [x_{n-1},x_n), \, n=1,2, \dots, N, \\
g_N(1) & \text{if} \quad x = 1,\\
g(x-1) & \text{if} \quad x \in (1,\infty). \end{cases}\]
\end{theorem}

\begin{proof}
It follows immediately from the fact that the attractor of the IFS with functions as in Equation~\eqref{eq:I1} that 
\[ \begin{aligned}
\left \{ (x,f(x)) \, : \, x \in [0,1] \right \} &= \bigcup_{n=1}^N \left \{ (L_n(x), a f(x) + g_n(x) \, : \, x\in [0,1] \right \}  \\
&= \bigcup_{n=1}^N \left \{ (x, a  f(L_n^{-1}(x)) + g_n(L_n^{-1}(x)) \, : \, x\in \left [\frac{N-1}{n},\frac{n}{N}\right ] \right \}.
\end{aligned} \]
This implies, for $x \in [(n-1)/N, n/N], n = 1,2, \dots, N$ and in the space $L^{\infty}(\R) \cap L^2([0,1]) \cap \mathcal P$, that 
\[f(x) = a f(Nx - (n-1)) + g_n(L_n^{-1}(x)) = a f(Nx) + g(x).\]
\end{proof}

Let  $g \in L^{\infty}(\R) \cap L^2([0,1]) \cap \mathcal P$, the space of statement (6) Corollary~\ref{cor:key},
and assume that $g$ has  the following properties
\begin{enumerate}
\item  $g$ is continuous on the intervals $[x_0,x_1], (x_1,x_2], \dots , (x_{N-1}, x_N]$,
\item  the limit from the right $g(\frac{n}{N}+)$ exists for $n= 1, \dots, N-1$.
\end{enumerate}

\begin{theorem}  \label{thm:interp2}
Let $f$ be the unique solution to the functional equation $f(x) - af(Nx) = g(x)$
considered in the space $L^{\infty}(\R) \cap L^2([0,1]) \cap \mathcal P$ of Corollary~\ref{cor:key}, where
$g$ has properties (1-2) above.

Then $f$ interpolates the data $\{(x_n,y_n), n=0, 1,2,\dots, N\}$, where $x_n = n/N$ and
\[ \begin{aligned}
y_0 &= g(0)/(1-a) \\
y_N &= g(1)/(1-a) \\
y_n &= g(x_n) + \frac{a}{1-a}\, g(1), \; n=1,2, \dots, N-1.
\end{aligned} \]
Moreover, the closure of the graph of $f$ restricted to the domain $[0,1]$ is the unique attractor of the IFS
 $W=([0,1]\times\mathbb{R};\,w_{1} ,w_{2},\dots,w_{N})$, where
$w_{n}(x,y)=(L_{n}(x),ay+g_{n}(x)),\,n=1,2,\dots ,N$, and
\[ \begin{aligned}
L_n(x) &= (x+n-1)/N \\
g_n(x) &= \begin{cases} g\big(L_n(x)\big) & \text{if} \quad 0<x<1 \\
g \left ( \frac{n-1}{N}+ \right )  & \text{if} \quad x=0 \\
g\left ( \frac{n}{N} \right )& \text{if} \quad x=1.
\end{cases}
\end{aligned}
\]
\vskip 2mm
If, in addition to properties (1-2) of the function $g$, we have
\[ (3) \quad g \left (\frac{n}{N}+ \right ) - g \left (\frac{n}{N}\right ) = \frac{a}{1-a} \, \left (g(1)-g(0) \right) \]
for $n=1,2,\dots, N-1$,
then $f$ is continuous on $[0,1]$, and the graph of $f$ restricted to the domain $[0,1]$ is the unique attractor of the
above IFS.
\end{theorem}
\begin{proof}
Concerning the interpolation of the data, assume that $\left\vert a\right\vert <1$. Statement (1) of
Corollary~\ref{cor:key} guarantees a unique solution given by $f(x)=\sum
_{k=0}^{\infty}\,a^{k}\,g(N^{k}x)$. Substituting $x=0$ into the functional
equation, we obtain $f(0)-af(0)=g(0)$ which implies $f(0)=\frac{g(0)}{1-a}=y_0$.
Substituting $x=1$ in the series expansion yields
\[
f(1)=\sum_{k=0}^{\infty}\,a^{k}\,g( N^{k} ) = \sum_{k=0}^{\infty}\,a^{k}\,g(1)
= y_{N}.
\]
With $1\leq n\leq N-1$, substituting $x=x_{n}$ and using
properties of $g$, we have
\[
\begin{aligned}
f(x_n)  &= \sum_{k=0}^{\infty}\,a^{k}\,g( N^k x_n) =\sum_{k=0}^{\infty}\,a^{k}\,g \left ( N^k \frac{n}{N}\right )
= g\left ( \frac{n}{N} \right ) +\sum_{k=1}^{\infty}\,a^{k}\,g  ( N^{k-1} n)  \\ &=
g\left ( \frac{n}{N} \right ) +\sum_{k=1}^{\infty}\,a^{k}\,g(1) = g(x_n) + a y_N = y_n.
\end{aligned}
\]

Concerning the statement about the closure of the graph of $f$, we consider the following set-valued map associated with the IFS $W$.  With a slight abuse of notation, we shall denote the associated map  also with $W$ and let $W:2^{[0,1]\times
{\mathbb{R}}}\rightarrow2^{[0,1]\times{\mathbb{R}}}$ defined by
\[
W(B)=\bigcup_{i=1}^{N}w_{i}(B),
\]
Let $G:=\left\{  (x,f(x))\,:\,x\in\lbrack0,1]\right\}  $. It is well known,
see for example \cite[Theorem 3.2]{continuations}, that under the stated
conditions the IFS $W$ possesses a unique attractor. The attractor is the
unique compact set $A\subset\lbrack0,1]\times{\mathbb{R}}$ such that $W(A)=A$.
It suffices to show that $W(\overline{G})=\overline{G}$. Note that
 $f$ is periodic with period $1$.

To show that $W(\overline G) = \overline G$, we first show that
\begin{equation}
\label{eq:contain1}W(\widehat{G} ) \subseteq G,
\end{equation}
where $\widehat{G} = G \setminus\{(0,f(0)), (1,f(1))\}$. For any $n=1,2,\dots,
N$, let $(x^{\prime},y^{\prime}) \in w_{n}(\widehat{G})$, Then there is an
$(x,y)$ such that $x\in[0,1] \setminus\{0,1\}, \, y = f(x), \, x^{\prime}=
L_{n}(x) = (n-1+x)/N$, and
\[
y^{\prime}= ay+g_{n}(x) = a\, f(x) + g(L_{n}(x)) = a\, f(Nx^{\prime}-n+1) +
g(L_{n}(L_{n}^{-1}(x^{\prime}))) = a\,f(Nx^{\prime}) + g(x^{\prime}).
\]
This implies that $y^{\prime}= f(x^{\prime})$, so that $(x^{\prime},y^{\prime
}) \in G$.

We next show that
\begin{equation}
\label{eq:contain2}\widetilde{G} \subseteq W(G),
\end{equation}
where $\widetilde{G} = G \setminus\{(n/N,f(n/N)), \, n = 0,1,2, \dots, N\}$.
Assume that $(x,y) \in\widetilde{G}$ and, without loss of generality, that $x
\in((n-1)/N,n/N)$. Let $x^{\prime}= L_{n}^{-1}(x), \, y^{\prime}= f(x^{\prime
})$. Then
\[
y = f(x) = a\,f(Nx)+g(x) = a\, f(x^{\prime}+N-1) + g(L_{n}(x^{\prime})) = a\,
f(x^{\prime}) + g_{n}(x^{\prime}).
\]
Therefore $(x,y) = w_{n}(x^{\prime},y^{\prime}) \in W(G)$.

Note that the map $w_{n}\,: \, [0,1]\times\mathbb{R}\rightarrow\lbrack(n-1)/n,n/N] \times \R$
is a homeomorphism. From Equation~\eqref{eq:contain1} and
Equation~\eqref{eq:contain2}, respectively,
\[
\begin{aligned}
W(\overline G) &= W(\overline {\widehat G}) = \overline{W(\widehat G}) \subseteq \overline G \\
\overline G &= \overline {\widetilde G} \subseteq \overline {W(G)} = W(\overline G).
\end{aligned}
\]

With the additional assumption (3) we have, $F_1(x_0,y_0)=y_0$ and for $n=2,3, \dots, N$
\[ \begin{aligned} F_{n}(x_0,y_0) &= ay_0 + g_n(0) = ay_0 + g\left (\frac{n-1}{N}+ \right ) = ay_0 +g \left (\frac{n-1}{N}\right )  + \frac{a}{1-a} \left (g(1)-g(0) \right )
=y_{n-1}. \end{aligned}\]
Similarly $F_n(x_N,y_N) =y_n$.
Therefore the functions $F_n$ satisfy Equation~\eqref{eq:I3}, in  which case the attractor of the IFS
is the graph of a continuous function.
\end{proof}

The present formalism allows  both continuous and discontinuous
interpolants, as illustrated in Example \ref{anexample}, in contrast to
continuous interpolants in the traditional theory of fractal interpolation
functions. Furthermore, the FIF obtained herein can be evaluated
pointwise to desired precision, by summing absolutely and uniformly convergent
series. We note that discontinuous fractal functions are also mentioned in
\cite{manuvascues1}.

\begin{example} \label{anexample} It follows from Theorems~\ref{thm:interp1} and \ref{thm:interp2} that the attractor $A\subset\mathbb{[}0,1]\times\mathbb{[}-1,1]$ of the
contractive IFS
\[
W=\{\mathbb{R}^{2};w_{1}(x,y)=(x/2,ay),w_{2}(x,y)=(x/2+1/2,(1-a)+ay)\}\text{,}
\]
where $-1<a<1$,  is the closure of the graph, restricted to the domain $[0,1]$, of the unique function $f$ in the space
$L^{\infty}(\R) \cap L^2([0,1]) \cap \mathcal P$
that is the solution to the equation
\[f(x) - a f(2x) = g(x),\]
where
\[
g(x)=\begin{cases}
0 & \text{for }x\in\lbrack0,1/2]\\
1-a & \text{for }x\in(1/2,1]\\
g(x-n) &\text{for }x\in(n,n+1], \; n\in\mathbb{N}\text{.}
\end{cases}
\]
Moreover, the function $f$ interpolates the  data $\{(0,0),(0.5,a),(1,1)\}$.
The function $g:[0,\infty)\rightarrow\mathbb{R}$ is not continuous on $[0,1]$.  The function $f:[0,1]\rightarrow\mathbb{R}$, that can
be represented by the uniformly and absolutely convergent series
\[
f(x)=\sum\limits_{k=0}^{\infty}a^{k}g(2^{k}x),
\]
is not continuous for $a \neq 1/2$. That the function $f$ is discontinuous on $[0,1]$ for $a \neq 1/2$ can be verified, for instance, by showing that $f(1/2^+)=f(1/2)$ and $f(0^+)=f(0)$ cannot  be satisfied simultaneously.
When $a= 1/2$, the function $f$ is continuous; in fact $f(x)=x$.
\end{example}

For a formulation more closely related to continuous  fractal interpolation functions, as illustrated in the next
paragraph, one may assume
\begin{equation}
\label{eq:fg}g(x):=g_{0}(x)+f_{0}(x)-af_{0}(Nx)
\end{equation}
where $f_{0}:\mathbb{R}\rightarrow\mathbb{R}$ is such that
\[
\begin{aligned} f_{0}(0)&=y_{0} \\
f_{0}(1)&=y_{N} \\
f_{0}(x)&=f_{0}(x-1) \quad \text{ for} \; \; x\in(1,\infty), \end{aligned}
\]
and $g_{0}:\mathbb{R\rightarrow R}$ is continuous and such that
\[
\begin{aligned} g_{0}(0)&=g_{0}(1)=0 \\ g_{0}(x) &=g_{0}(x+1) \quad \text{for all} \;\; x\in\mathbb{R} \\
g_{0}(x_{n}) &=y_{n}-f_{0}(x_{n}) \quad \text{for} \;\;  n=1,2,\dots, N-1. \end{aligned}
\]
It is readily confirmed that $g(x)$ satisfies the conditions (1)-(2) of Theorem \ref{thm:interp2}.
Furthermore, it can be shown that, if $f_{0}(x)$ is continuous on $[0,1]$
(from the right at $x=0$ and from the left at $x=1$), then $g$ satisfies condition (3) prescribed in Theorem \ref{thm:interp2}.
Consequently, the solution $f(x)$
to the functional equation $f(x)-af(Nx)=g(x)$ is continuous on $[0,1]$. Note,
however, that it is not in general assumed that $g(x)$ is continuous.

In this setting, the free parameters, namely the \textquotedblleft base function" $f_{0}$, the function
$g_{0},$ and the vertical scaling parameter $a$, may be chosen to obtain
diverse fractal interpolation systems, for instance, Hermite and spline FIFs
\cite{Barnsley1,CK,Nav1}. They can also be chosen to control the Minkowski
dimension and other properties of the graph of the approximant $f$. For
example, it is reported in \cite{baranski2} that both the Minkowski dimension
and the packing dimension of the graph of $f$ are given by $D=\max
\{2+ \frac{\ln\left\vert a\right\vert}{\ln N},1\}$, for various classes of function
$g_{0}$. Consistent formulas for the Minkowski dimensions related to graphs
of a fractal interpolation function are established in \cite{della,geronimo,hardin}.

In references \cite{Barnsley1, Nav2} it is observed that  the notion of fractal interpolation can be used to associate an entire family of fractal functions $\{h^\alpha: \alpha \in (-1,1)^{N}\}$ with a prescribed continuous function $h$ on a compact interval. To this end, one may consider Equation (\ref{eq:I2}) with
$g_n(x) = h \big(L_n(x)\big) -\alpha_n b(x)$, where $b: [0,1] \to \mathbb{R}$ is a continuous function such that $b \not \equiv h$ and $b$ interpolates $h$ at the extremes of the interval $[0,1]$. Each function $h^\alpha$  in this family is referred to as $\alpha$-fractal function or \textquotedblleft fractal perturbation" corresponding to $h$. In our present setting, the function $f$ is the fractal perturbation corresponding to $g_{0}+f_{0}$ with base function $f_{0}$
and constant scale vector $\alpha$ whose components are $a$. Therefore, the
$\alpha$-fractal function and the approximation classes obtained through the
corresponding fractal operator (see, for instance, \cite{Nav2,Vis}) can also
be discussed using the present formalism.

\section{Weierstrass Fourier Approximation}

\label{sec:ON}

This section deals with a framework for a ``fractal" Fourier analysis. A
natural complete orthonormal basis set of fractal functions is provided that
serves as a rough analog of the standard sine-cosine Fourier basis. These
fractal counterparts are obtained as solutions $f$ to the functional
equation~\eqref{eqref}, with $g\in\{\sin2k\pi x,\cos2k\pi x\}_{k=1}^{\infty} \cup \{1\}$.

\begin{proposition}
\label{prop:ON} Let $f(x)$ be the solution to $f(x)-af(bx) = g(x)$ in
$L^{2}({\mathbb{R}})$ where $\left\vert a\right\vert <\left\vert b\right\vert
^{1/2}$. If $\left\{  g_{k}\right\}  _{k=1}^{\infty}$ is an orthononormal
basis for $L^{2}(\mathbb{R)}$, then
\[
\langle f_{k},f_{l}\rangle=c+\sum_{n=1}^{\infty}\frac{a^{2n}}{b^{n}}
\,\sum_{m=1}^{n}\frac{b^{m} }{a^{m}}\langle g_{k},(T_{b^{m}}+T_{b^{m}}^{\ast
})g_{l}\rangle,
\]
where $c=(1-a^{2}/b)^{-1}$.
\end{proposition}

\begin{proof}
Define $T_{a,b}=aT_{b}$. We have $T_{a,b}^{\ast}=T_{\frac{a}{b},\frac{1}{b}}$,
and also $T_{a,b}T_{c,d}=T_{ac,bd}=T_{c,d}T_{a,b}.$ On taking the product,
term-by-term, of two absolutely and uniformly convergent series of linear
operators, we obtain
\begin{align*}
((I-T_{a,b})^{\ast}(I-T_{a,b}))^{-1}  &  =(I-T_{a,b})^{-1}(I-T_{a,b}^{\ast
})^{-1}\\
&  =(I-T_{a,b})^{-1}(I-T_{\frac{a}{b},\frac{1}{b}})^{-1}\\
&  =\left(  \sum_{n=0}^{\infty}\,a^{n}T_{b}^{n}\right)  \left(  \sum
_{m=0}^{\infty}\,\left(  \frac{a}{b}\right)  ^{m}T_{\frac{1}{b}}^{m}\right) \\
&  =\sum_{n=0}^{\infty}\sum_{m=0}^{\infty}\,\frac{a^{n+m}}{b^{m}}T_{b}
^{n}T_{\frac{1}{b}}^{m}=\sum_{n=0}^{\infty}\sum_{m=0}^{\infty}\,\frac{a^{n+m}
}{b^{m}}T_{b^{n-m}}
\end{align*}
Now let $\left\{  g_{k}\right\}  _{k=1}^{\infty}$ be an orthononormal basis
for $L^{2}(\mathbb{R)}.$ Let $f_{k}=M_{a,b}^{-1}(g_{k})=(I-T_{a,b})^{-1}g_{k}
$. Since, by Theorem \ref{thm:key}, $(I-T_{a,b})^{-1}$ is a linear
homeomorphism on $L^{2}(\mathbb{R)}$, the set of functions $\left\{
f_{k}\right\}  _{k=1}^{\infty}$ is a Riesz basis for $L^{2}(\mathbb{R)}$.
Then
\begin{align*}
&  \langle f_{k},f_{l}\rangle=\langle g_{k},((I-T_{a,b})^{\ast}(I-T_{a,b}
))^{-1}g_{l}\rangle\\
&  =\sum_{n=0}^{\infty}\sum_{m=0}^{\infty}\,\frac{a^{n+m}}{b^{m}}\langle
g_{k},T_{b^{n-m}}g_{l}\rangle\\
&  =\sum_{n=0}^{\infty}\left(  \frac{a^{2}}{b}\right)  ^{n}+\sum
_{\substack{n,m=0\\m<n}}^{\infty}\,\frac{a^{n+m}}{b^{m}}\langle g_{k}
,T_{b^{n-m}}g_{l}\rangle+\sum_{\substack{n,m=0\\m>n}}^{\infty}\,\frac{a^{n+m}
}{b^{m}}\langle g_{k},T_{b^{n-m}}g_{l}\rangle\\
&  =c+\sum_{\substack{n,m=0\\m<n}}^{\infty}\,\frac{a^{n+m}}{b^{m}}\langle
g_{k},T_{b^{n-m}}g_{l}\rangle+\sum_{\substack{n,m=0\\m<n}}^{\infty}
\,\frac{a^{n+m}}{b^{n}}\langle g_{k},T_{b^{m-n}}g_{l}\rangle\\
&  =c+\sum_{\substack{n,m=0\\m<n}}^{\infty}\,\frac{a^{n+m}}{b^{m}}\langle
g_{k},T_{b^{n-m}}g_{l}\rangle+\sum_{\substack{n,m=0\\m<n}}^{\infty}
\,\frac{a^{n+m}}{b^{m}}\langle g_{k},T_{b^{n-m}}^{\ast}g_{l}\rangle\\
&  =c+\sum_{n=1}^{\infty}\,\sum_{m=0}^{n-1}\frac{a^{n+m}}{b^{m}}\langle
g_{k},(T_{b^{n-m}}+T_{b^{n-m}}^{\ast})g_{l}\rangle\\
&  =c+\sum_{n=1}^{\infty}\,\sum_{m=1}^{n}\frac{a^{2n-m}}{b^{n-m}}\langle
g_{k},(T_{b^{m}}+T_{b^{m}}^{\ast})g_{l}\rangle\\
&  =c+\sum_{n=1}^{\infty}\frac{a^{2n}}{b^{n}}\,\sum_{m=1}^{n}\frac{b^{m}
}{a^{m}}\langle g_{k},(T_{b^{m}}+T_{b^{m}}^{\ast})g_{l}\rangle
\end{align*}
where $c=(1-a^{2}/b)^{-1}$.
\end{proof}

A similar looking but different expression obtains in the case $\left\vert
a\right\vert >\left\vert b\right\vert ^{1/2}$. Clearly, such series are
amenable to computation, as we illustrate in the next section. For another
example, the $g_{k}$s in Proposition~\ref{prop:ON} could be $(\sqrt{\pi}
2^{k}k!)^{-\frac{1}{2}}H_{k}(x)\exp(-x^{2}/2)$, where the $H_{k}$s are Hermite
polynomials \cite{johnston}.

\subsection{Weierstrass Fourier Basis}

Working in $L^{2}[0,1]$, the inner product is $\langle f,h\rangle
:=\int\limits_{0}^{1}f(x)h(x)dx$. The set of functions $\{\sqrt{2}\cos k2\pi
x\}_{k=1}^{\infty} \, \cup\, \{\sqrt{2}\sin k2\pi x\}_{k=1}^{\infty}\cup\{1\}$
is a complete orthonormal basis for $L^{2}[0,1].$ Consider these as functions
on $\mathbb{R}$, periodic of period $1$. Let
\begin{align*}
c_{k}(x)  &  =\sqrt{2}\cos k2\pi x,\text{ }\\
s_{k}(x)  &  =\sqrt{2}\sin k2\pi x\text{, }\\
e(x)  &  =1\text{,}
\end{align*}
for all $k\in\mathbb{N}$ and $x\in\mathbb{R}$. Inner products are given by
\begin{align*}
&  \langle s_{k},s_{l}\rangle=\langle c_{k},c_{l}\rangle=\delta_{k,l},\\
&  \langle s_{k},c_{l}\rangle=\langle e,c_{k}\rangle=\langle e,s_{k}
\rangle=0,\text{ }\langle e,e\rangle=1,
\end{align*}
for all $k,l\in\mathbb{N}$.

Let $b=2,\left\vert a\right\vert <1$, and $M=M_{a,b}$. In view of statement
(6) of Corollary~\ref{cor:key}, a new normalized basis for $L^{2}[0,1]$ is
$\{\widehat{e},\widehat{c}_{k},\widehat{s}_{k}:k\in\mathbb{N}\}$, where
\[
\begin{aligned}
\widehat{e} &=(1-a) \, M ^{-1}(e)=e \\
\widehat{c}_{k} &=\sqrt{1-a^{2}}\, M^{-1}(c_{k}) \\
\widehat{s}_{k}&=\sqrt{1-a^{2}}\, M^{-1}
(s_{k}).
\end{aligned}
\]
For $1\leq k\leq l$, the inner products are
\begin{align*}
&  \langle\widehat{c}_{k},\widehat{c}_{l}\rangle=2(1-a^{2})\sum_{n,m=0}
^{\infty}a^{n+m}\int\limits_{0}^{1}(\cos k\pi2^{n+1}x)(\cos l\pi2^{m+1}x)dx\\
&  =(1-a^{2})\sum_{n,m=0}^{\infty}a^{n+m}\delta_{2^{n}k,2^{m}l}=(1-a^{2}
)\sum_{\substack{n,m=0\\n\geq m}}^{\infty}a^{n+m}\delta_{2^{n}k,2^{m}l}\\
&  =(1-a^{2})\sum_{\substack{n,m=0\\n\geq m}}^{\infty}a^{n+m}\delta
_{2^{n-m}k,l}=(1-a^{2})\sum_{i,m\geq0}a^{2m+i}\delta_{2^{i}k,l}\\
&  =\left\{
\begin{array}
[c]{cc}
a^{j} & \text{if }l=2^{j}k,\\
& \\
0 & \text{otherwise.}
\end{array}
\right.
\end{align*}
Similar expressions are obtained for $\{\widehat{s}_{k}\}_{k=1}^{\infty}$. In
summary, for all $k,l\in\mathbb{N}$,
\begin{equation}
\begin{aligned} \langle\widehat{c}_{k},\widehat{e}\rangle &=\langle\widehat{s}_{k},\widehat{e}\rangle=\langle \widehat{c} _{k},\widehat{s}_{l}\rangle=0 \; \; \text{ and }\; \langle\widehat{e},\widehat{e}\rangle=1, \\ \langle \widehat{c}_{k},\widehat{c}_{l} \rangle = \langle \widehat{s}_{k},\widehat{s}_{l} \rangle &=\begin{cases} a^{j}\quad \text{ if }k=2^{j}l\text{ or }l=2^{j}k\text{ for some }j\in \mathbb{N} \cup \{0\}, \\ 0 \quad \; \text{ if }k\neq2^{j}l\text{ and }l\neq2^{j}k\text{ for some }j\in\mathbb{N} \cup \{0\}. \end{cases} \end{aligned} \label{eq:cs}
\end{equation}
The Gram matrix of inner products of these basis functions, as displayed
below, is relatively sparse. See the work of Per-Olof L\"{o}wdin on overlap
matrices in quantum mechanics, for example \cite{lowdine}.
\[
\left(  \langle\widehat{s}_{k},\widehat{s}_{l}\rangle\right)  _{k,l=1}
^{\infty}=\left(  \langle\widehat{c}_{k},\widehat{c}_{l}\rangle\right)
_{k,l=1}^{\infty}=\left(
\begin{array}
[c]{ccccccccc}
1 & a^{1} & 0 & a^{2} & 0 & 0 & 0 & a^{3} & .\\
a^{1} & 1 & 0 & a^{1} & 0 & 0 & 0 & a^{2} & .\\
0 & 0 & 1 & 0 & 0 & a^{1} & 0 & 0 & .\\
a^{2} & a^{1} & 0 & 1 & 0 & 0 & 0 & a^{1} & .\\
0 & 0 & 0 & 0 & 1 & 0 & 0 & 0 & .\\
0 & 0 & a^{1} & 0 & 0 & 1 & 0 & 0 & .\\
0 & 0 & 0 & 0 & 0 & 0 & 1 & 0 & .\\
a^{3} & a^{2} & 0 & a^{1} & 0 & 0 & 0 & 1 & .\\
. & . & . & . & . & . & . & . & .
\end{array}
\right)  .
\]
Note that, for $m=0,1,2,3$,
\[
\det\left(  \langle\widehat{c}_{k},\widehat{c}_{l}\rangle\right)
_{k,l=1}^{2^{m}}=(1-a^{2})^{2^{m}}\text{,}
\]
which suggests that this formula holds for all $m\in\mathbb{N}\cup\{0\}$.

The graph of each function $\widehat{c}_{k},\widehat{s}_{k}$ has Minkowski (and in ``many
cases" Hausdorff) dimension $D=2+\left(  \ln a\right)  /\ln2$ when $a>0.5$;
see \cite{baranski} and \cite{hunt}. It is straightforward to apply the
Gram-Schmidt algorithm to obtain the complete orthonormal (ON) basis of Weierstrass
nowhere differentiable functions given in the following theorem.

\begin{theorem}
\label{thm:ONbasis} The set of functions $\{1, \widetilde{c}_{k},
\widetilde{s}_{k} : k \in\mathbb{N}\}$, where
\[
\begin{aligned}
\widetilde c_i &= \begin{cases} \widehat c_i & \text{if} \; i \; \text{is odd} \\
\frac{\widehat c_i - a \, \widehat c_{i/2}}{\sqrt{1-a^2}} = \sqrt{1-a^2} \; \widehat c_i - a \,c_{i/2} & \text{if} \; i \; \text{is even}\end{cases}
\\ \\
\widetilde s_i &= \begin{cases} \widehat s_i & \text{if} \; i \; \text{is odd} \\
\frac{\widehat s_i - a \, \widehat s_{i/2}}{\sqrt{1-a^2}} =  \sqrt{1-a^2} \; \widehat s_i - a \,s_{i/2} & \text{if} \; i \; \text{is even},\end{cases}
\end{aligned}
\]
is a complete ON basis for $L^{2}([0,1])$.
\end{theorem}

\begin{proof}
Using the relations in Equation~\eqref{eq:cs} it follows readily that, if $k$
is odd and $l$ is even, then $\langle\widetilde{c}_{k},\widetilde{c}_{l}
\rangle= 0$ unless $l = k2^{j}$ for some positive integer $j$, in which case,
\[
\sqrt{1-a^{2}} \, \langle\widetilde{c}_{k},\widetilde{c}_{l} \rangle=
\langle\widehat{c}_{k},\widehat{c}_{l} \rangle- a \, \langle\widehat{c}
_{k},\widehat{c}_{l/2} \rangle= a^{j} - a \, a^{j-1} = 0.
\]
For $k < l$, both even, it again readily follows that $\langle\widetilde{c}
_{k},\widetilde{c}_{l} \rangle= 0$ unless $l = k2^{j}$ for some positive
integer $j$, in which case,
\[
(1-a^{2}) \, \langle\widetilde{c}_{k},\widetilde{c}_{l} \rangle=
\langle\widehat{c}_{k},\widehat{c}_{l} \rangle+ a^{2} \, \langle
\widehat{c}_{k/2},\widehat{c}_{l/2}\rangle- a \, \langle\widehat{c}
_{k},\widehat{c}_{l/2} \rangle- a \, \langle\widehat{c}_{k/2},\widehat{c}_{l}
\rangle= a^{j} + a^{j+2} - a\, a^{j-1} - a\, a^{j+1} = 0.
\]
For $k=l$, both even,
\[
(1-a^{2}) \, \langle\widetilde{c}_{k},\widetilde{c}_{k} \rangle=
\langle\widehat{c}_{k},\widehat{c}_{k} \rangle+ a^{2} \, \langle
\widehat{c}_{k/2},\widehat{c}_{k/2}\rangle- 2a \, \langle\widehat{c}
_{k},\widehat{c}_{k/2} \rangle= 1 + a^{2} - 2a \, a = 1-a^{2}.
\]

To show the equality of the two expressions in the even cases, express
$\widehat{c}_{i}$ (or $\widehat{s}_{i}$) as a sum of the $c_{i}$'s (or $s_{i}
$'s) using Equation~\eqref{eq:sum} and simplify.
\end{proof}

A given function $h\in L^{2}([0,1])$ has a Fourier expansion in terms of the
complete ON basis $\{1,s_{k},c_{k}:k\in\mathbb{N}\}$. If $h$ is, in addition,
bounded and extended periodically, it has an expansion, that we refer to as a
\textbf{Weierstrass Fourier series}, in terms of the complete ON basis
$\{1,\widetilde{s}_{k},\widetilde{c}_{k}:k\in\mathbb{N}\}$ of fractal functions.

\begin{theorem}
If $h \in L^{2}([0,1])$ has Fourier expansion
\[
h(x) = \alpha_{0} + \sum_{n=1}^{\infty} \left[  \alpha_{n} \, c_{n}(x) +
\beta_{n} \, s_{n}(x) \right]  ,
\]
then on the interval $[0,1]$ it also has Weierstrass Fourier expansion
\[
h(x) = \widetilde{\alpha}_{0} + \sum_{n=1}^{\infty} \left[  \widetilde{\alpha
}_{n}\, \widetilde{c}_{n}(x) + \widetilde{\beta}_{n} \, \widetilde{s}_{n}(x)
\right]  ,
\]
where $\widetilde{\alpha}_{0} = \alpha_{0}$ and
\[
\begin{aligned} \widetilde \alpha_n &= \begin{cases} \sqrt{1-a^2} \, \sum_{m=0}^{\infty} a^m \, \alpha_{n2^m} & \text{if $n$ is odd} \\ \\
- a \, \alpha_{n/2} + (1-a^2) \, \sum_{m=0}^{\infty} a^m \, \alpha_{n2^m}  & \text{if $n$ is even}, \end{cases} \\ \\
\widetilde  \beta_n &= \begin{cases} \sqrt{1-a^2} \, \sum_{m=0}^{\infty} a^m \, \beta_{n2^m} & \text{if $n$ is odd} \\ \\
- a \, \beta_{n/2} + (1-a^2) \, \sum_{m=0}^{\infty} a^m \, \beta_{n2^m}  & \text{if $n$ is even}.
\end{cases} \end{aligned}
\]

\end{theorem}

\begin{proof}
To compute $\widetilde{\alpha}_{n}=\langle h,\widetilde{c}_{n}\rangle$,
express $\widetilde{c}_{n}$ and $\widetilde{s}_{n}$ in terms of the
$\widehat{c}_{n}$ and $c_{n}$ using Theorem~\ref{thm:ONbasis}, then just in
terms of the $c_{n}$ using Equation~\eqref{eq:sum}. The orthogonality
relations for the respective sine and cosine functions yields the formulas in
the statement of the theorem, similarly for the computation of
$\widetilde{\beta}_{n}=\langle h,\widetilde{s}_{n}\rangle$.
\end{proof}

\begin{remark}  If $a=0$, then $\widetilde c_n = c_n$ and $\widetilde s_n = s_n$, for all $n$, and
 the Weierstrass Fourier series reduces to the classical Fourier series.
\end{remark}

\begin{example}
\label{ex2}Figures \ref{Fig10terms}, \ref{figure50terms} and
\ref{figure50terms(point3)} illustrate both classical Fourier and Weierstrass
Fourier approximations of the function $h(x)=x-0.5$ over the interval $[0,1]$.
\end{example}

\begin{figure}[ptb]
\centering
\fbox{\includegraphics[height=3in, keepaspectratio
%natheight=14.221800in,
%natwidth=14.221800in,
%height=3.5993in,
%width=3.5993in
]
{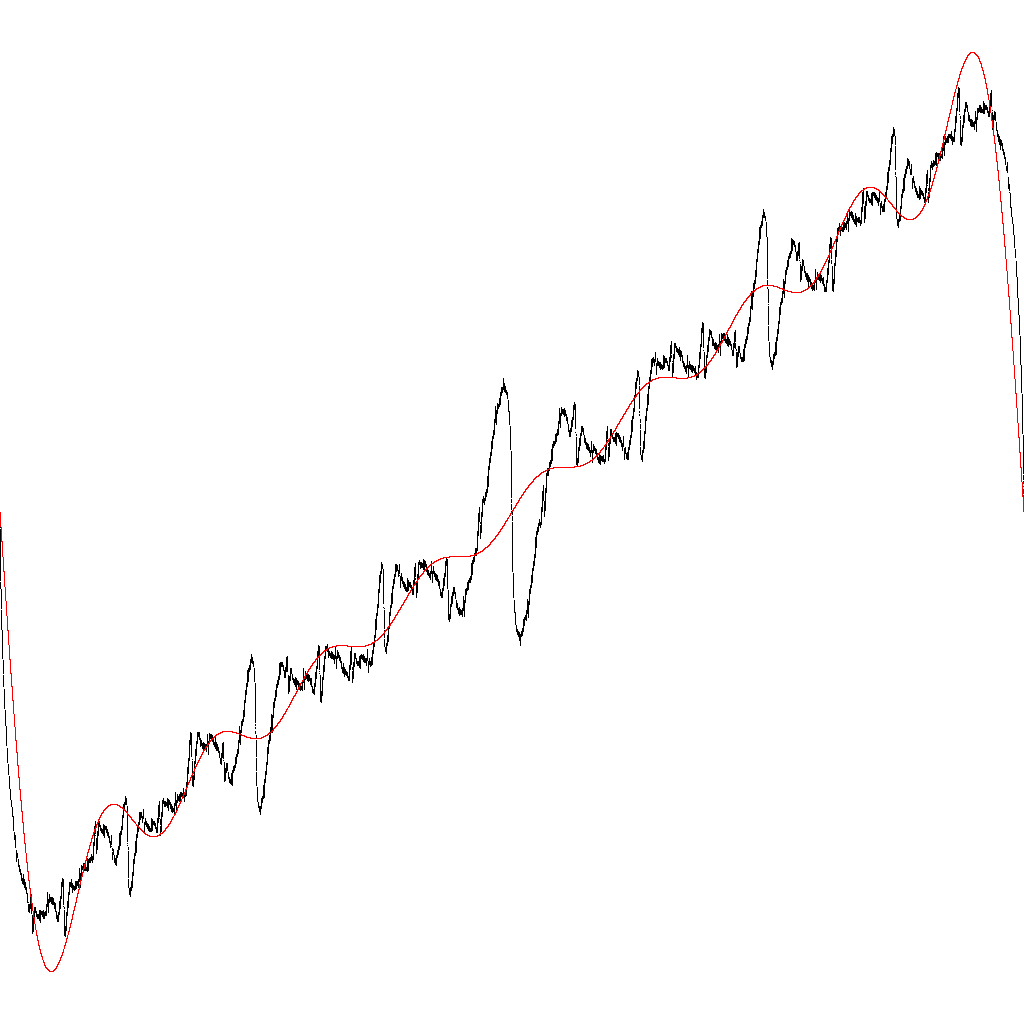}
}\caption{The sum of the first ten terms of the Fourier (red) and the
Weierstrass Fourier ($a=0.6$) series (black) approximations of the function
$h(x)=x-0.5$.}
\label{Fig10terms}
\end{figure}

\begin{figure}[ptb]
\centering
\fbox{\includegraphics[height=3in, keepaspectratio
%natheight=14.221800in,
%natwidth=14.221800in,
%height=3.5993in,
%width=3.5993in
]
{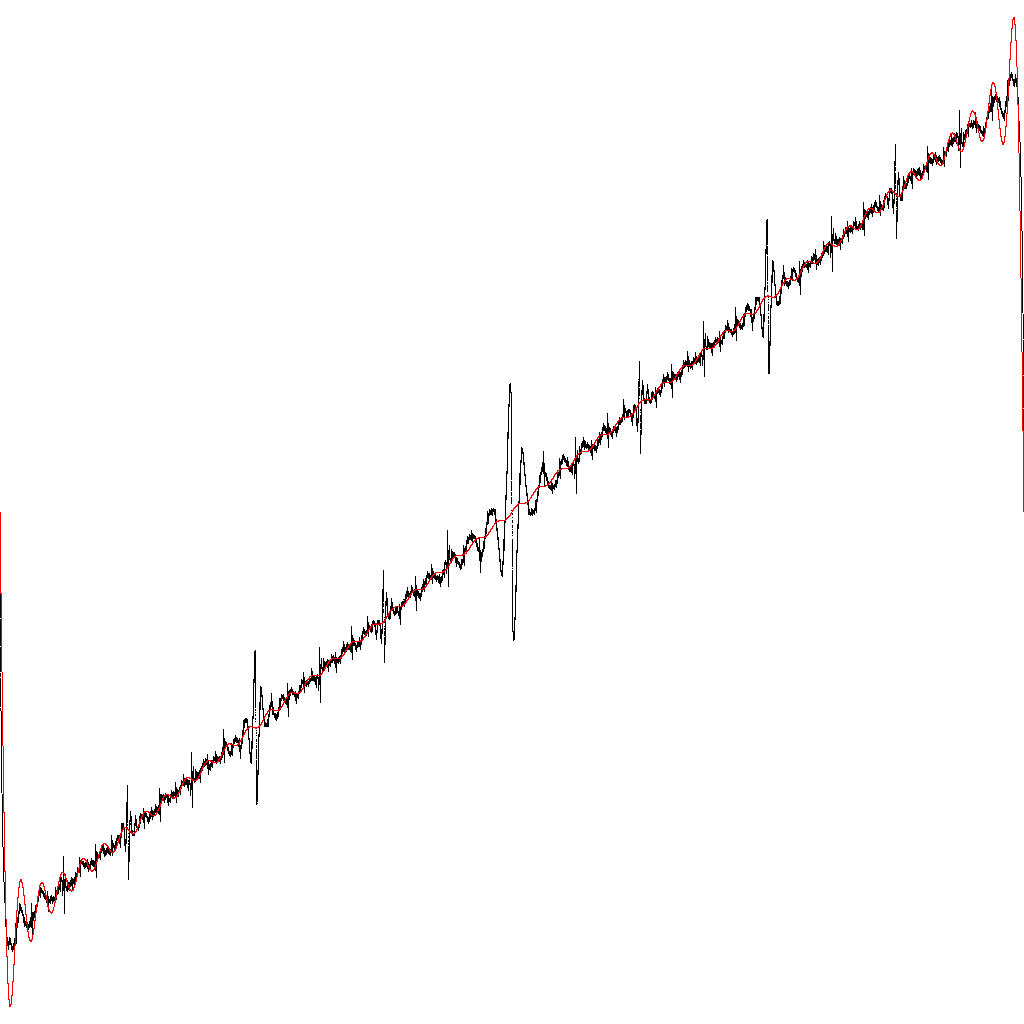}
}\caption{The sum of the first fifty terms of the Fourier (red) and the
Weierstrass Fourier ($a=0.6$) series (black) approximations of the function
$h(x)=x-0.5$.}
\label{figure50terms}
\end{figure}

\begin{figure}[ptb]
\centering
\fbox{\includegraphics[height=3in, keepaspectratio
%natheight=14.221800in,
%natwidth=14.221800in,
%height=3.5993in,
%width=3.5993in
]
{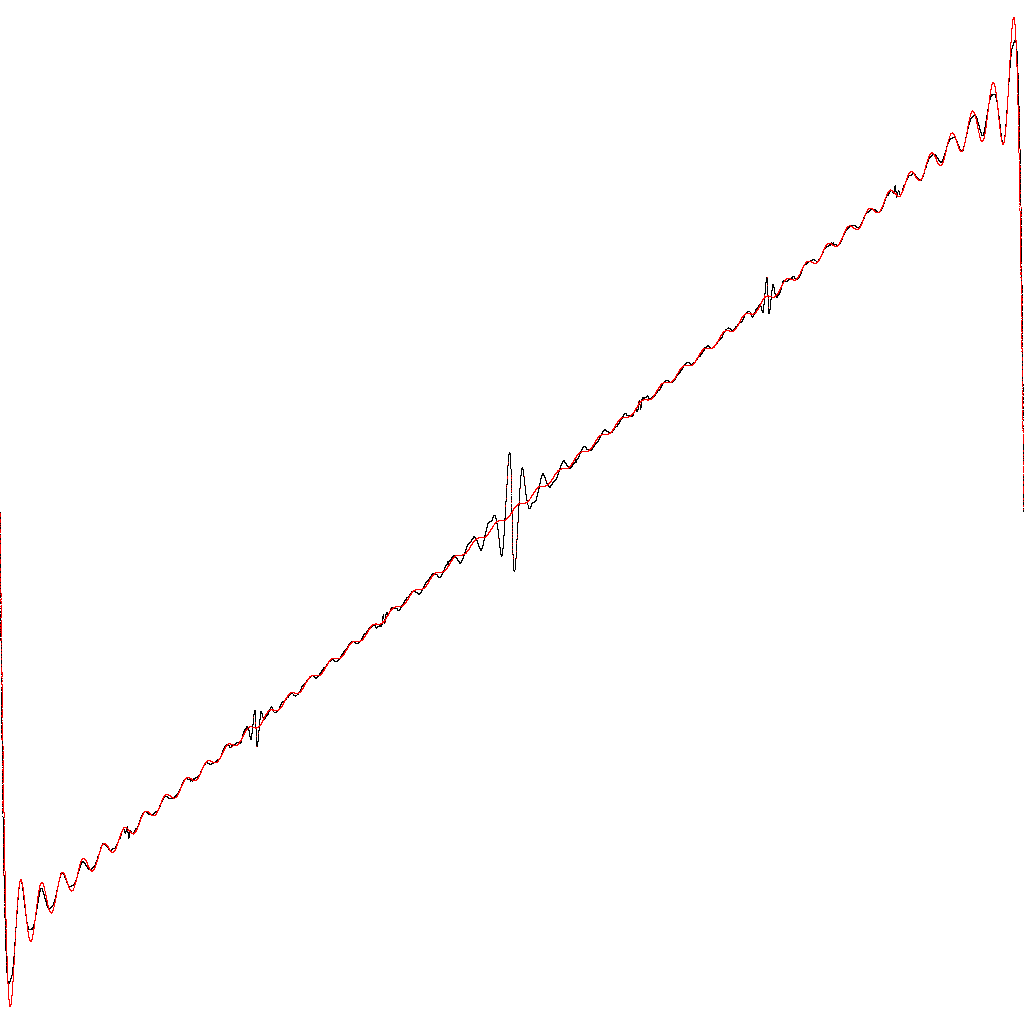}
}\caption{The sum of the first fifty terms of the Fourier (red) and the
Weierstrass Fourier ($a=0.3$) series (black) approximations of the function
$h(x)=x-0.5$.}
\label{figure50terms(point3)}
\end{figure}

\begin{example}
Other examples, using a discretized version of the theory and both theoretical
and experimental data, are reported in \cite{zhang}. In one example, a
discretized version of Example \ref{ex2} with $a=0.5$, the $L^{2}$ errors,
obtained by subsampling both the approximants and $h(x)$ at $512$ equally
spaced points, were compared: it was found that the Weierstrass Fourier series
performed slightly better than the classical Fourier series, for all partial sums
of length $l$ for $l=1,2,...,510$. In some other examples, the performance was
worse, as measured by the $L^{2}$ error.
\end{example}

\begin{remark}[Error Analysis]
In various spaces, such as $L^{2}[0,1]$, the finite Fourier sum $\sum
C_{n}c_{n}(x)+\sum S_{n}s_{n}(x)+Ee(x)$ is close to $g(x)$ if and only if the
corresponding Weierstrass Fourier sum $\sum C_{n}\widehat{c}_{n}(x)+\sum
S_{n}\widehat{s}_{n}(x)+E\widehat{e}(x)$ is close to $M_{a,b}^{-1}g(x)$. While
the errors remain under control, the smoothness of functions, as measured by
their differentiability and box-counting dimensions, can be altered. There is
a huge literature, and a good understanding, of error issues for classical
Fourier analysis. A future direction of research is to derive the Weierstrass
Fourier analogues based on the ON basis $\{\widetilde{c}_{n},\widetilde{s}
_{n}$,$\widetilde{e}\}$ instead of the classical basis $\{c_{n},s_{n},e\}$. It
may thus be possible to include deterministic roughness in approximation and
interpolation procedures.
\end{remark}

\end{document}